\documentclass[hidelinks,11pt]{article} 
\title{Unidimensional factor models imply weaker partial correlations than zero-order correlations.
}

\usepackage[a4paper]{geometry}
\usepackage{graphicx}
\usepackage[format=hang]{caption}
\usepackage{amsmath}

\usepackage{amsthm}
\newtheorem*{assump}{Assumptions}
\newtheorem*{prop*}{Proposition}
\newcommand*{\LargerCdot}{\raisebox{-0.25ex}{\scalebox{1.5}{$\cdot$}}}
\usepackage{amssymb}
\usepackage{fancyhdr}
\usepackage{authblk}

\usepackage{datetime}
\newdateformat{mydate}{\monthname[\THEMONTH] \THEYEAR}
\mydate

\let\OLDthebibliography\thebibliography
\renewcommand\thebibliography[1]{
  \OLDthebibliography{#1}
  \setlength{\parskip}{5pt}
  \setlength{\itemsep}{5pt plus 0.3ex}
}

\usepackage{hyperref}
\usepackage[usenames,dvipsnames]{color}
\hypersetup{
  colorlinks,
  citecolor=Blue,
  linkcolor=Blue,
  urlcolor=Blue
}
\usepackage[comma, sort,longnamesfirst]{natbib}
\setcitestyle{semicolon}

\begin{document}
\author{Riet van Bork, Raoul P. P. P. Grasman, and Lourens J. Waldorp}
\affil{University of Amsterdam}
\maketitle 
\pagestyle{fancy}
\begin{abstract}
In a unidimensional factor model it is assumed that the set of indicators that loads on this factor are conditionally independent given the latent factor. Two indicators are, however, never conditionally independent given (a set of) other indicators that load on this factor, as this would require one of the indicators that is conditioned on to correlate one with the latent factor. Although partial correlations between two indicators given the other indicators can thus never equal zero \citep{holland1986conditional}, we show in this paper that the partial correlations do need to be always weaker than the zero-order correlations.
More precisely, we prove that the partial correlation between two observed variables that load on one factor given all other observed variables that load on this factor, is always closer to zero than the zero-order correlation between these two variables. 
\end{abstract}

The unidimensional factor model plays an important role in, among others, psychometrics, educational measurement and sociology. In these disciplines the unidimensional factor model is used to measure a construct by assuming that the shared variance among a set of indicators reflects its common cause. For example, by assuming that a set of IQ items is caused by intelligence, one can use the shared variance among the testscores on these items as an estimate of intelligence. So far multiple implications of the unidimensional factor model have been put forward that can be used to evaluate the fit of a factor model to observed data. Probably the most fundamental condition of the unidimensional factor model is that the indicators are conditionally independent given the latent factor. This condition is sometimes called \emph{latent conditional independence} \citep{holland1986conditional} but is more widely known as \emph{local independence} in item response theory \citep{lord1980applications}. Another important implication of the unidimensional factor model is that the model implied cavariance matrix, which is a matrix of rank one, implies that some so-called tetrads equal zero.  A tetrad refers to the difference between the products of two pairs of covariances among four random variables \citep{bollen1993confirmatory}. The tetrads that equal zero are called vanishing tetrads and by testing whether the tetrads in a sample covariance matrix differ from zero, this implication of factor models can be used to evaluate model misspecification. In this paper we define another implication of the unidimensional factor model: the partial correlation between two indicators given all other indicators is always closer to zero than the zero-order correlation between these indicators.

\section{Unidimensional factor models}

We first define unidimensional factor models before providing a proof for the implication of unidimensional factor models that partial correlations are always weaker than their corresponding zero-order correlations. Let $\Sigma$ denote the covariance matrix of $\mathbf{y}$, in which  $\mathbf{y}$ denotes the vector of observed variables. We assume that $\Sigma$ is nondegenerate. Let $\lambda$ denote the vector of factor loadings and $\Theta$ the residual covariance matrix. In a unidimensional factor model all observed variables in $\mathbf{y}$ are a linear function of the same factor, $\eta$, and independent residuals, $\varepsilon$:
\begin{equation}
y_i = \lambda_{i}\eta +\varepsilon_i.
\end{equation}
We assume var($\eta$) = 1. We also assume independent residuals, that is, $\forall i\ne j(\varepsilon_i \perp \varepsilon_j)$ 
which implies that $\Theta$ is diagonal. In the following, we assume that standardized factor loadings range from -1 to 1 but are never exactly -1 or 1, as this would imply that the common factor is observed. We also assume that factor loadings are not exactly zero, as this would imply that the corresponding observed variable is not an indicator of the common factor. Put differently, we assume $\forall i((-1<\lambda_i<0) \lor (0<\lambda_i<1))$. The model implied covariance matrix of the indicators is a function of the factor loadings and the residual covariance matrix:
\begin{equation}
\Sigma = \lambda\lambda' +\Theta.\label{eq:UFMSigma}
\end{equation}
Equation (\ref{eq:UFMSigma}) implies that the covariance among observed variables is a function of their factor loadings. More precisely, because $\Theta$ is a diagonal matrix, the covariance between two variables $y_i$ and $y_j$ equals $\lambda_{i}\lambda_{j}$, in which $\lambda_{i}$ and $\lambda_{j}$ are elements of the vector $\lambda$ and denote the factor loadings of $y_i$ and $y_j$ on the factor $\eta$.

Consider three variables $y_i$, $y_j$ and $y_z$. The partial correlation between $y_i$ and $y_j$ given $y_z$ is defined as follows \citep{chen2014graphical}:

\begin{equation}
\rho_{ij\LargerCdot z} = \frac{\rho_{ij}-\rho_{iz}\rho_{jz}}{\sqrt{(1-\rho^2_{iz})(1-\rho^2_{jz})}}\label{eq:pcor}
\end{equation}
Some structures of three correlations imply that the partial correlation is stronger than the correlation. For example when $\rho_{jz}$ is negative while the other two correlations are positive, results in a partial correlation $\rho_{ij.z}$ that is stronger than the zero order correlation $\rho_{ij}$. \citet{langford2001property} showed that the property of being positively correlated is not transitive and thus for three variables it is possible to have a correlational structure with one negative and two positive correlations. However, a structure with one negative and two positive correlations is not possible under a unidimensional factor model as it is impossible to choose factor loadings such that they result in two positive correlations and one negative correlation. In fact all possible correlational structures that result in some partial correlations that are stronger than their corresponding zero-order correlation appear to be impossible under a unidimensional factor model. For example, for three variables $y_1$, $y_2$ and $z$ we can substitute each correlation in equation (\ref{eq:pcor}) with factor loadings:
\begin{equation}\label{pcor3}
\rho_{ij\LargerCdot z}=\frac{(\lambda_{y_1}\lambda_{y_2})- (\lambda_{y_1}\lambda_z)(\lambda_{y_2}\lambda_z)}{\sqrt{ 1-(\lambda_{y_1}\lambda_z)^2}\sqrt{ 1-(\lambda_{y_2}\lambda_z)^2}}=\frac{\lambda_{y_1}\lambda_{y_2}- \lambda_{y_1}\lambda_{y_2}\lambda_z^2}{\sqrt{ 1-\lambda_{y_1}^2\lambda_z^2}\sqrt{ 1-\lambda_{y_2}^2\lambda_z^2}}
\end{equation}
Since $\lambda_z^2$ is a number between zero and one, the partial correlation $\rho_{ij\LargerCdot z}$ is weaker than the zero-order correlation $\lambda_{y_1}\lambda_{y_2}$. The following section provides a proof of the implication of unidimensional factor models that all partial correlations between indicators are necessarily weaker than the zero-order correlations between these indicators.

\section{Weaker partial correlations than zero-order correlations}
We start with providing the assumptions defined in the the previous section.
\begin{assump}\begin{minipage}[t]{\linewidth}
\begin{enumerate}
\item $\Sigma$ is nondegenerate.
\item Residuals are independent, i.e., $\forall i\ne j(\varepsilon_i \perp \varepsilon_j)$.
\item  The variance of $\eta$ equals one, i.e., var($\eta$)=1.
\item  All standardized factor loadings do not equal zero, one or\\ minus one, i.e., $\forall i((-1<\lambda_i<0) \lor (0<\lambda_i<1))$.
\end{enumerate}
\end{minipage}
\end{assump}
\noindent Based on these assumptions we define the following proposition that we prove in this paper.
\begin{prop*}
Assume 1 to 4 above. For a set of $p$ Gaussian random variables that load on one common factor, the partial correlation between two of these variables while the other $p-2$ variables are partialled out, is closer to zero than the zero-order correlation between the two variables.
\end{prop*}
\begin{proof}
We first prove that the proposition holds for positive correlations and then prove that the above proposition also holds for negative correlations. This means that for positive correlations we show that the partial correlation is smaller than the zero-order correlation and for negative correlations we show that the partial correlation is larger than the zero-order correlation.

We use the following formula\footnote{Equation (\ref{general}) is an application of the well known Schur complement \citep{zhang2006schur}; the conditional covariance of $Y$ given $Z$ (i.e., $\Sigma_{YY\bullet Z}$) is the Schur complement of $\Sigma_{ZZ}$.} for the partial covariance matrix \citep{johnson1998applied}:
\begin{equation}\label{general}
\Sigma_{YY\bullet Z}= \Sigma_{YY} - \Sigma_{YZ} \Sigma_{ZZ}^{-1} \Sigma_{ZY}
\end{equation}
$Y=[y_1,y_2]^T$ in which $y_1$ and $y_2$ denote two arbitrary variables from the set of $p$ observed variables that load on the common factor. $Z=[z_1, z_2,\hdots, z_{p-2}]^T$, denoting all other $p-2$ variables that load on the same common factor as $Y$. Since we assume that all variables load on a common factor we can define all of these matrices in terms of the factor loadings. After all, the unidimensional factor model implies $\rho_{y_1y_2} = \lambda_{y_1}\lambda_{y_2}$.
\begin{equation}
\Sigma_{YY}=\begin{bmatrix}
       1  & \lambda_{y_1}\lambda_{y_2}           \\[0.3em]
       \lambda_{y_1}\lambda_{y_2} & 1 
     \end{bmatrix}
\end{equation}
Let $\textbf{z}$ denote the vector of factor loadings of the $p-2$ variables that are partialled out $[\lambda_{z_1}, \lambda_{z_2},\hdots, \lambda_{z_{p-2}}]^T $. 
\begin{equation}
\Sigma_{ZY}= (\lambda_{y_1},\lambda_{y_2})^T\textbf{z}^T = \Sigma_{YZ}^T
\end{equation}
\begin{equation}
\Sigma_{ZZ}=\textbf{z}\textbf{z}^T + A
\end{equation}
in which $A$ is a diagonal matrix with the residual variances of $Z$ on the diagonal (i.e., $A=\text{diag}(1-\lambda_{z_1}^2, 1-\lambda_{z_2}^2,\hdots, 1-\lambda_{z_{p-2}}^2)$). As such $\Sigma_{ZZ}$ is the correlation matrix of the variables that are partialled out.
 
Note that in equation (\ref{general}) $\Sigma_{YY\bullet Z}$ is the $2\times2$ partial covariance matrix of $y_1$ and $y_2$ and needs to be standardized to obtain the partial correlation matrix. The partial correlation between $y_1$ and $y_2$ is thus the off-diagonal element of $\Sigma_{YY\bullet Z}$ divided by the square root of the diagonal elements of $\Sigma_{YY\bullet Z}$.
From this it follows that the partial correlation between $y_1$ and $y_2$ given the variables $\{z_1, z_2,\hdots, z_{p-2}\}$ equals:

\begin{equation}\label{pcor}
\frac{\lambda_{y_1}\lambda_{y_2}- \lambda_{y_1}\lambda_{y_2}\textbf{z}^T\Sigma_{ZZ}^{-1}\textbf{z}}{\sqrt{ 1-\lambda_{y_1}^2\textbf{z}^T\Sigma_{ZZ}^{-1}\textbf{z}}\sqrt{ 1-\lambda_{y_2}^2\textbf{z}^T\Sigma_{ZZ}^{-1}\textbf{z}}}
\end{equation}
Note that in the simple case where $p=3$, Z consists of one variable, $\Sigma_{ZZ}^{-1}=1$ and thus $\textbf{z}^T\Sigma_{ZZ}^{-1}\textbf{z}$ equals $\lambda_z^2$ resulting in equation (\ref{pcor3}).

To prove that the partial correlation is always weaker than the zero-order correlation we first prove that for a positive correlation between $y_1$ and $y_2$ the partial correlation is smaller than the zero-order correlation. To do so, we have to prove that for positive correlations the following inequality results in a contradiction:

\begin{equation}\label{positive}
\frac{\lambda_{y_1}\lambda_{y_2}- \lambda_{y_1}\lambda_{y_2}\textbf{z}^T\Sigma_{ZZ}^{-1}\textbf{z}}{\sqrt{ 1-\lambda_{y_1}^2\textbf{z}^T\Sigma_{ZZ}^{-1}\textbf{z}}\sqrt{ 1-\lambda_{y_2}^2\textbf{z}^T\Sigma_{ZZ}^{-1}\textbf{z}}} \geq \lambda_{y_1}\lambda_{y_2} > 0
\end{equation}

This can be proved by showing that $0<\textbf{z}^T\Sigma_{ZZ}^{-1}\textbf{z}<1$. First assume that $0<\textbf{z}^T\Sigma_{ZZ}^{-1}\textbf{z}<1$. If $\textbf{z}^T\Sigma_{ZZ}^{-1}\textbf{z}$ is a number between zero and one, then the numerator of the above fraction will always be between zero and the correlation (i.e., $\lambda_{y_1}\lambda_{y_2}$). Thus, for a positive correlation the numerator will be positive. As a result in the equation above we can take out $\lambda^2_{y_1}$ and $\lambda^2_{y_2}$ from the square root in the denominator of the fraction and change the $\geq$ sign in a $>$ sign. After all, for a positive correlation, taking out $\lambda^2_{y_1}$ and $\lambda^2_{y_2}$ makes the total fraction larger. Then we are left with: 
\begin{equation}
\frac{\lambda_{y_1}\lambda_{y_2}(1-\textbf{z}^T\Sigma_{ZZ}^{-1}\textbf{z})}{ 1-\textbf{z}^T\Sigma_{ZZ}^{-1}\textbf{z}}=\lambda_{y_1}\lambda_{y_2}>\lambda_{y_1}\lambda_{y_2}> 0
\end{equation}
This is a contradiction; $\lambda_{y_1}\lambda_{y_2}$ is not larger than $\lambda_{y_1}\lambda_{y_2}$, and so we conclude that the partial correlation is strictly smaller than the zero-order correlation for positive correlations. 

Now we show that for negative correlations the partial correlation is always larger than the zero-order correlation. To do so, we have to prove that for a negative correlation between $y_1$ and $y_2$ the following inequality results in a contradiction:

\begin{equation}\label{negative}
\frac{\lambda_{y_1}\lambda_{y_2}- \lambda_{y_1}\lambda_{y_2}\textbf{z}^T\Sigma_{ZZ}^{-1}\textbf{z}}{\sqrt{ 1-\lambda_{y_1}^2\textbf{z}^T\Sigma_{ZZ}^{-1}\textbf{z}}\sqrt{ 1-\lambda_{y_2}^2\textbf{z}^T\Sigma_{ZZ}^{-1}\textbf{z}}} \leq \lambda_{y_1}\lambda_{y_2} < 0
\end{equation}
Again this can be proved by showing that $0<\textbf{z}^T\Sigma_{ZZ}^{-1}\textbf{z}<1$. First assume that $0<\textbf{z}^T\Sigma_{ZZ}^{-1}\textbf{z}<1$. If $\textbf{z}^T\Sigma_{ZZ}^{-1}\textbf{z}$ is a number between zero and one, then the numerator of the above fraction will always be between zero and the correlation (i.e., $\lambda_{y_1}\lambda_{y_2}$). Thus, for a negative correlation the numerator will be negative. As a result in the equation above we can take out $\lambda^2_{y_1}$ and $\lambda^2_{y_2}$ from the square root in the denominator of the fraction and change the $\leq$ sign in a $<$ sign. After all, for a negative correlation, taking out $\lambda^2_{y_1}$ and $\lambda^2_{y_2}$ makes the total fraction smaller. Then we are left with: 
\begin{equation}
\frac{\lambda_{y_1}\lambda_{y_2}(1-\textbf{z}^T\Sigma_{ZZ}^{-1}\textbf{z})}{ 1-\textbf{z}^T\Sigma_{ZZ}^{-1}\textbf{z}}=\lambda_{y_1}\lambda_{y_2}<\lambda_{y_1}\lambda_{y_2}< 0
\end{equation}
This is a contradiction; $\lambda_{y_1}\lambda_{y_2}$ is not smaller than $\lambda_{y_1}\lambda_{y_2}$. We thus have proven that if $0<\textbf{z}^T\Sigma_{ZZ}^{-1}\textbf{z}<1$ then both for cases in which the zero-order correlation between $y_1$ and $y_2$ is positive and for cases in which the zero-order correlation between $y_1$ and $y_2$ is negative, the partial correlation between $y_1$ and $y_2$ is closer to zero than the zero-order correlation between $y_1$ and $y_2$.

Now we prove that $0<\textbf{z}^T\Sigma_{ZZ}^{-1}\textbf{z}<1$. Note that $\Sigma_{ZZ} = \mathbf{zz}^T + A$, so that:

\begin{equation}
\textbf{z}^T\Sigma_{ZZ}^{-1}\textbf{z} = \textbf{z}^T[\textbf{z}\textbf{z}^T + A]^{-1}\textbf{z}.
\end{equation}
Note that since $A$ is the residual covariance matrix of $Z$, $A$ is a diagonal nondegenerate matrix. 
The Sherman-Morrison formula \citep{sherman1950adjustment} gives us:
\begin{equation}
\textbf{z}^T[\textbf{z}\textbf{z}^T + A]^{-1}\textbf{z}=\textbf{z}^T\Big[A^{-1}-\frac{A^{-1}\textbf{z}\textbf{z}^TA^{-1}}{1+\textbf{z}^TA^{-1}\textbf{z}}\Big]\textbf{z}
\end{equation}
This equals:
\begin{equation}
\textbf{z}^TA^{-1}\textbf{z}-\frac{(\textbf{z}^TA^{-1}\textbf{z})^2}{1+\textbf{z}^TA^{-1}\textbf{z}}=\frac{\textbf{z}^TA^{-1}\textbf{z}}{1+\textbf{z}^TA^{-1}\textbf{z}}<1
\end{equation}
We assumed that $\Sigma$ is positive definite, thus $\Sigma_{ZZ}$ is positive definite, and thus $\Sigma_{ZZ}^{-1}$ is positive definite as well. Therefore by definition $0<\textbf{z}^T\Sigma_{ZZ}^{-1}\textbf{z}$.
As a consequence,  $0<\textbf{z}^T\Sigma_{ZZ}^{-1}\textbf{z}<1$ and as such, both equation (\ref{positive}) and equation (\ref{negative}) result in a contradiction. This proves that in a  unidimensional factor model the partial correlation is always weaker than the zero-order correlation.
\end{proof}
In this paper we proved that the unidimensional factor model implies that partial correlations are always weaker than their corresponding zero-order correlations. Additionally the proof offers a relatively simple way to calculate the model implied partial correlation between two variables that load on a factor, given all other variables that load on this factor. One way to obtain this partial correlation is by standardizing the inverse of the model implied covariance matrix ($\hat{\Sigma}^{-1}$) and multiply the off-diagonal elements with -1. That is, the model implied partial correlations are a function of the model implied precision matrix, $\hat{P} = \hat{\Sigma}^{-1}$ \citep{whittaker2009graphical}:  
\begin{equation} \label{eq:partialcor}
\hat{\rho}_{ij \LargerCdot Z } =  - \frac{\hat{p}_{ij}}{\sqrt{\hat{p}_{ii}\hat{p}_{jj}}}
\end{equation}
In this paper we showed that if the model is a unidimensional factor model, the model implied partial correlation between two variables loading on this factor equals:
\begin{equation}
\hat{\rho}_{ij \LargerCdot Z } =  \frac{\lambda_{y_1}\lambda_{y_2}-\lambda_{y_1}\lambda_{y_2}K}{\sqrt{1-\lambda^{2}_{y_1}K}\sqrt{1-\lambda^{2}_{y_2}K}}
\end{equation}
In which $K$ is a single number between zero and one, and is a function of the factor loadings of the variables that are partialled out:
\begin{equation}
K=\frac{\textbf{z}^TA^{-1}\textbf{z}}{1+\textbf{z}^TA^{-1}\textbf{z}}
= \frac{\sum z_i^2/a_{ii}}{1 + \sum z_i^2/a_{ii}}
\end{equation}
To obtain $K$ we still need to calculate the inverse of the matrix $A$, but since $A$ is a diagonal matrix this is more easily calculated than the inverse of $\Sigma$ which is needed for equation (\ref{eq:partialcor}).

\section*{Acknowledgements}
We would like to thank Mijke Rhemtulla and Denny Borsboom for their help in constructing the theory that later resulted in this proof. We would like to thank Sacha Epskamp for his helpful comments.

\bibliography{references}
\bibliographystyle{apalike}
\end{document}